\documentclass[a4paper,12pt,twoside]{article}
\usepackage[latin2]{inputenc}
\usepackage{amssymb}
\usepackage{amsmath}
\usepackage[all,2cell,dvips]{xy}
\usepackage[mathscr]{eucal}
\usepackage[dvips]{graphicx}
\makeatletter

\def\Lifts{{\cal L}}

\def\Forb{\mathop{\mathrm{Forb_h}}\nolimits}
\def\Forbi{\mathop{\mathrm{Forb_e}}\nolimits}
\def\relS{R_{\mathbf{S}}}
\def\Age{\mathop{\mathrm{Age}}\nolimits}
\def\Aut{\mathop{\mathrm{Aut}}\nolimits}
\def\Rel{\mathop{\mathrm{Rel}}\nolimits}
\def\CSP{\mathop{\mathrm{CSP}}\nolimits}

\def\relsys#1{\mathbf {#1}}

\def\extsys#1{\mathbf #1'}
\def\rel#1#2{R_{\mathbf{#1}}^{#2}}

\def\ext#1#2{X_{\mathbf{#1}}^{#2}}

\def\F{{\cal F}}

\def\M{{\cal M}}

\def\K{{\cal K}}

\def\Piece{{\cal P}}

\def\U{{\cal U}}
\def\UU{{\relsys U}}
\def\sh{\mathop{\mathrm{Sh}}\nolimits}

\def\Fraisse{Fra\"{\i}ss\' e}

\newtheorem{defn}{Definition}[section]

\newtheorem{corollary}{Corollary}[section]
\newtheorem{observation}{Observation}[section]
\newtheorem{prop}{Proposition}[section]

\newtheorem{thm}{Theorem}[section]
\newtheorem{lem}[thm]{Lemma} 

\DeclareRobustCommand{\qed}{%
  \ifmmode \mathqed
  \else
    \leavevmode\unskip\penalty9999 \hbox{}\nobreak\hfill
    \quad\hbox{$\square$}%
  \fi
}
\let\QED@stack\@empty
\let\qed@elt\relax
\newcommand{\pushQED}[1]{%
  \toks@{\qed@elt{#1}}\@temptokena\expandafter{\QED@stack}%
  \xdef\QED@stack{\the\toks@\the\@temptokena}%
}
\newcommand{\popQED}{%
  \begingroup\let\qed@elt\popQED@elt \QED@stack\relax\relax\endgroup
}
\def\popQED@elt#1#2\relax{#1\gdef\QED@stack{#2}}
\newcommand{\qedhere}{%
  \begingroup \let\mathqed\math@qedhere
    \let\qed@elt\setQED@elt \QED@stack\relax\relax \endgroup
}

\providecommand{\proofname}{Proof}
\newenvironment{proof}[1][\proofname]{\par
  \pushQED{\qed}%
  \normalfont \topsep6\p@\@plus6\p@\relax
  \trivlist
  \item[\hskip\labelsep \bf {#1\ignorespaces.}]\ignorespaces
}{%
\popQED\endtrivlist
\par
}

\begin{document}
\bibliographystyle{plain}

\title{Homomorphism and embedding universal structures for restricted classes}

\author{
   \normalsize {\bf Jan Hubi\v cka$^*$}\\
   \normalsize {\bf Jaroslav Ne\v set\v ril}\thanks {The Computer Science Institute of Charles University (IUUK) is supported by grant ERC-CZ LL-1201 of the Czech Ministry of Education and CE-ITI P202/12/G061 of GA\v CR.}\\
{\small Computer Science Institute of Charles University (IUUK)}\\
{\small and}\\
{\small Institute of Theoretical Computer sciences (ITI)}\\
   {\small Charles University}\\
   {\small Malostransk\' e n\' am. 25, 11800 Praha, Czech Republic}\\
   {\normalsize 118 00 Praha 1}\\
   {\normalsize Czech Republic}\\
   {\small \{hubicka,nesetril\}@iuuk.ms.mff.cuni.cz}
}

\date{}
\maketitle
\begin{abstract}
This paper unifies problems and results related to
(embedding) universal and homomorphism universal structures.  On the one side
we give a new combinatorial proof of the existence of universal objects for
homomorphism defined classes of structures (thus reproving a result of
Cherlin, Shelah and Shi) and on the other side this leads to the new proof of
the existence of dual objects (established by Ne{\v{s}}et{\v{r}}il and Tardif). Our explicite
approach has further applications to special structures such as variants of
the rational Urysohn space.  We also solve a related extremal problem which shows the optimality (of the used lifted arities) of our construction (and a related problem of A.~Atserias).
\end{abstract}
\section {Introduction}
It is an old mathematical idea  to reduce a study of a particular class of objects to a certain single ``universal'' object. It is hoped that this object may then  be used to study the given (infinite) set of individual problems in a more systematic and perhaps even efficient way. For example the universal object may have interesting additional properties (such as symmetries and ultrahomogeneity) which in turn can be used to classify the finite problems.
A good example of this is the classification of Ramsey classes of finite structures via the classification program of generic structures \cite{Nevsetvril1989a,Nevsetvril1995,Kechris2005,Hubicka2005a}. In this paper we deal with
universal objects from the homomorphism point of view. We define and study homomorphism-universal (shortly hom-universal \cite{Nesetril1989})
objects and show their relationship to embedding-universal (shortly universal) objects:

Given a class $\K$ of countable structures, an object $\UU\in\K$ is called {\em hom-universal} (or {\em universal}) if for every object $\relsys{A}\in \K$ there exists a homomorphism (or an embedding) $\relsys{A}\to \UU$.

The characterisation of those classes $\K$ which have an universal (or hom-universal) object is a well known  open problem which was studied intensively, see e.g. \cite{Komjath1999,Komjath1988,Cherlin1996,Nesetril1989,Nesetril2000}. The whole area was inspired by the negative results (see \cite{Hajnal1984,Cherlin1994}): for example the class of graphs not containing $C_l$ (= cycle of length $l$) fails to be universal for any $l>3$.  For universal structures the strongest results in the positive direction were obtained by Cherlin, Shelah and Shi in \cite{Cherlin1999}. Particularly, they proved the following (as consequence of Theorem 4 \cite{Cherlin1999}):

\begin{thm}
\label{1thm}
For every finite set $\relsys{F}_1, \relsys{F}_2, \ldots,\relsys{F}_t$ of finite
connected graphs  the class
$\Forb(\relsys{F}_1, \relsys{F}_2,\ldots,\relsys{F}_t)$ has an universal object.
\end{thm}

Here $\Forb(\relsys{F}_1, \relsys{F}_2, \ldots, \relsys{F}_t)$ denotes the class of all countable graphs (and more generally relational structures) $\relsys{G}$ for which there is no homomorphism $\relsys{F}_i\to \relsys{G}$ for every $i=1,2,\ldots, t$. Formally $\Forb(\relsys{F}_1,\relsys{F}_2,\ldots,\relsys{F}_t)= \{\relsys{G}; \relsys{F}_i\nrightarrow \relsys{G}\mathrm{~for~}i=1,2,\ldots,t\}$.
(In this paper the structures are denoted by bold letters. This applies also to graphs. But a cycle of length $l$ is still denoted by $C_l$ and complete graph on $k$ vertices by $K_k$.)
Similarly by $\Forbi(\relsys{F}_1, \relsys{F}_2, \ldots, \relsys{F}_t)$ we
denote the class of all graphs (and more generally relational structures)
$\relsys{G}$ for which there is no embedding $\relsys{F}_i\to \relsys{G}$ for
every $i=1,2,\ldots, t$.  (Note that the \cite{Cherlin1999} proves a stronger
result:  The universal graph exists for any class with finite algebraic closure
and, moreover, that the existence of $\omega$-categorical universal graph is equivalent
to this condition.  This result was extended to relational structures in
\cite{Cherlin2001}. Our motivation comes from study of graph homomorphisms and thus
in this paper we are interested in this weaker version of their
result as presented by Theorem \ref{1thm}.)

The proof of Theorem \ref{1thm} given in \cite{Cherlin1999} is based on techniques of model theory and, quoting \cite{Cherlin1999}, no explicit universal object is constructed ``in any very explicite way''.  In this paper we give such a proof which is a streamlined version of proof given in \cite{Hubicka2013}. (\cite{Hubicka2013} contains a stronger result for regular, possibly infinite, families. Here we proceed in other direction proving the model companion and the optimality of expansion).

  Homomorphism universal structures lead to a very different area of Constraint Satisfaction Problems.
Clearly every universal object is also hom-universal.  This however does not hold conversely as shown by examples of classes all with of graphs bounded chromatic numbers. Another example is provided by the class of all planar graphs: $K_4$ is hom-universal for the class of planar graphs by virtue of the $4$-colour theorem while no universal graph exists \cite{Hajnal1984}. Of special interest are classes $\Forb(\relsys{F}_1, \relsys{F}_2, \ldots, \relsys{F}_t)$ which have finite hom-universal structure. Such hom-universal structures are called {\em duals} and in this context the classes with finite hom-universal objects were characterised \cite{Nesetril2000} as follows (for general finite relational structures)
\begin{thm}
 $\Forb(\relsys{F}_1, \relsys{F}_2, \ldots, \relsys{F}_t)$ has dual if and only if all structures $\relsys{F}_i$ are (relational) trees.
\end{thm}

In this paper we give a new construction of duals (Corollary \ref{nasdual}).

It is the aim of this paper to investigate both universal and hom-universal
objects in the context of a (seemingly much more restricted) class of generic
objects. Main result (Theorem \ref{univthm}) proves Theorem \ref{1thm} by means
of {\em shadows (reducts)} and {\em lifts (expansions)} of a generic (i.e.
ultrahomogeneous universal) structure. These classical model theoretic notions
will be, for the sake of completeness, briefly reviewed now (see e.g. \cite{Hodges1993}):

A {\em structure} $\relsys{A}$ is a pair $(A,(\rel{A}{i};i\in I))$ where $\rel{A}{i}\subseteq A^{\delta_i}$ (i.e. $\rel{A}{i}$ is a $\delta_i$-ary relation on $A$). The finite family $(\delta_i; i\in I)$ is called the {\em type} $\Delta$. The type is usually fixed and understood from the context. (Note that we consider relational structures only, and no function symbols.) If set $A$ is finite we call {\em $\relsys A$ finite structure}.   A {\em homomorphism} $f:\relsys{A}\to \relsys{B}=(B,(\rel{B}{i};i\in I))$ is a mapping $f:A\to B$ satisfying for every $(x_1,x_2,\ldots, x_{\delta_i})\in \rel{A}{i}\implies (f(x_1),f(x_2),\ldots,f(x_{\delta_i}))\in \rel{B}{i}$, $i\in I$. If $f$ is 1-1, then $f$ is called an {\em embedding}.
The class of all (countable) relational structures of type $\Delta$ will be denoted by $\Rel(\Delta)$. 

The class $\Rel(\Delta), \Delta=(\delta_i; i\in I)$, $I$ finite, is fixed throughout this paper. Unless otherwise stated all structures $\relsys{A}, \relsys{B},\ldots$ belong to $\Rel(\Delta)$.
Now let $\Delta'=(\delta'_i;i\in I')$ be a type containing type $\Delta$. (By this we mean $I\subseteq I'$ and $\delta'_i=\delta_i$ for $i\in I$.)
Then every structure $\relsys{X}\in \Rel(\Delta')$ may be viewed as structure $\relsys{A}=(A,(\rel{A}{i}; i\in I))\in \Rel(\Delta)$ together with some additional relations $\rel{X}{i}$ for $i\in I'\setminus I$. To make this more explicite these additional relations will be denoted by $\ext{X}{i}, i\in I'\setminus I$. Thus a structure $\relsys{X}\in \Rel(\Delta')$ will be written as:
$$\relsys{X}=(A,(\rel{A}{i};i\in I),(\ext{X}{i};i\in I'\setminus I))$$
and, by abuse of notation, briefly as:
$$\relsys{X}=(\relsys{A},\ext{X}{1},\ext{X}{2},\ldots, \ext{X}{N}).$$

 We call $\relsys{X}$ a {\em lift} of $\relsys{A}$ and $\relsys{A}$ is called the {\em shadow} of $\relsys{X}$. In this sense the class $\Rel(\Delta')$ is the class of all lifts of $\Rel(\Delta)$.
Conversely, $\Rel(\Delta)$ is the class of all shadows of $\Rel(\Delta')$.
 Note that a lift is also in the model theoretic setting called an {\em expansion} and a shadow is often called a {\em reduct}.
(Our terminology is motivated by a computer science context, see \cite{Kun2008}.)
We will use letters $\relsys{A}, \relsys{B}, \relsys{C}, \ldots$ for shadows (in $\Rel(\Delta)$) and letters $\relsys{X}, \relsys{Y}, \relsys{Z}$ for lifts (in $\Rel(\Delta'))$.

For lift $\relsys X=(\relsys{A}, \ext{X}{1},\ldots,\ext{X}{N})$ we denote by $\sh(\relsys{X})$ the relational
structure $\relsys{A}$ i.e. its shadow. ($\sh$
is called the {\em forgetful functor}.) Similarly, for a class $\K'$
of lifted objects we denote by $\sh(\K')$ the class of all shadows
of structures in $\K'$.
On the other hand for a class $\K$ of structures we denote by $\K'$ the class of lifted structures.

As it is well known, an universal object may be constructed by an iterated amalgamation of finite objects and this leads to a stronger notion of {\em $\K$-generic object}:
For a class $\K$ we say that an object $\relsys{U}$ is $\K$-generic if it is both $\K$-(embedding) universal and it is {\em ultrahomogeneous}.
The later notion means the following:
Every isomorphism $\varphi$ between two finite substructures $\relsys{B}$ and $\relsys{C}$ of $\relsys{U}$ may be extended to an automorphism of $\relsys{U}$.
The notion of ultrahomogeneous structure is one of the key notions of modern model theory and it is the source of the well known classification programme \cite{Lachlan1984}.

The ultrahomogeneous structures are characterised by the properties of finite substructures.
Denote by $\Age(\relsys{U})$ the class of all finite substructures of $\relsys{U}$.
For a class $\K$ of countable relational structures, we denote by $\Age(\K)$ the class of all finite structures isomorphic to a substructure of some $\relsys{A}\in\K$.
 By a following classical result \cite{Fraisse1986, Hodges1993} structure $\relsys{U}$ is ultrahomogeneous if and only if $\Age(\relsys{U})=\K$ is a countable class with the following three properties:

\begin{enumerate}
\item[(a)] (Hereditary property) For every $\relsys{A}\in \K$ and an induced substructure $\relsys{B}$ of $\relsys{A}$ we have $\relsys{B}\in \K$;
\item[(b)] (Joint embedding property) For every $\relsys{A}, \relsys{B}\in \K$ there exists $\relsys{C}\in \K$ such that both $\relsys{C}$ contains $\relsys{A}$ and $\relsys{B}$ as induced substructures;
\item[(c)] (Amalgamation property) For $\relsys{A},\relsys{B},\relsys{C}\in \K$ and $\varphi$ embedding of $\relsys{C}$ into $\relsys{A}$, $\varphi'$ embedding of $\relsys{C}$ into $\relsys{B}$, there exists $\relsys{D}\in \K$ and embeddings $\psi:\relsys{A}\to \relsys{D}$ and $\psi':\relsys{B}\to\relsys{D}$ such that $\psi\circ\varphi = \psi'\circ\varphi'$. See Figure \ref{amalgamfig}.
\begin{figure}
\centerline{\includegraphics{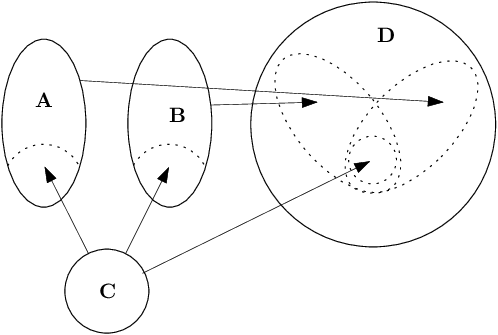}}
\caption{Amalgam of $\relsys{A}$ and $\relsys{B}$ over $\relsys{C}$.}
\label{amalgamfig}
\end{figure}
\end{enumerate}

Any class $\K$ satisfying the three conditions $(a),(b),(c)$ is called an {\em amalgamation class}.
In this case we denote by $\lim \K$ the (up to isomorphism uniquely determined) generic object $\relsys{U}$. $\lim \K=\relsys{U}$ is called the {\em \Fraisse{} limit} of $\K$.

In this paper we show that universal object for classes $\Forb(\F)$ may be constructed as a shadow (reduct) of a generic object of an explicitly defined lifted class. 
The following is the principal result of this paper:

\begin{thm}
\label{univthm}

For every finite family $\F$ of finite connected relational structures of type $\Delta$ there exists a type $\Delta'$ containing $\Delta$ and a finite family $\F'$ of finite structures of type $\Delta'$ such that:
\begin{enumerate}
\item $\Age(\Forbi(\F'))$ is an amalgamation class whose shadow is $\Age(\Forb(\F))$,
\item the shadow $\relsys{U}$ of the generic $\relsys{U}'=\lim \Forbi(\F')$ is universal structure for $\Forb(\F)$,
\item $\relsys{U}$ is model complete, and
\item the automorphism group of $\relsys{U}$ is the same as the automorphism group of $\relsys{U}'$. 
\end{enumerate}

\end{thm}
Among others this implies that $\relsys{U}'$ is ultrahomogeneous and $\relsys{U}$ is $\omega$-categorical (of course $\omega$-categoricity for universal objects for classes $\Forb(\F)$ was proved in \cite{Cherlin1999} as well.)

We prove Theorem \ref{univthm} in two steps: First we define class $\Lifts$ of lifted objects and then prove that $\Lifts$ is an amalgamation class (Theorem \ref{mainthm}).
Thus we can get universal objects for classes $\Forb(\F)$ in a particularly efficient way as shadow of a \Fraisse{} limit of certain (explicitly defined) lifted class $\Forbi(\F')$.   

Let us remark that the assumption that $\F$ is finite is not necessary for Theorem~\ref{univthm}. In \cite{Hubicka2013} we proved (by a different proof) an extension of Theorem~\ref{univthm} to some infinite (called regular) families $\F$. However in Section \ref{bounded} we use the finiteness of $\F$ to obtain lower bounds on arity of the lifts.

Theorem \ref{univthm} and its combinatorial proof have several other consequences. For example, we show that in the case that all $\relsys{F}\in \F$ are relational trees it suffices to consider only a class of monadic lifts (i.e. all relations $\ext{A}{i}$ for $i\in I'\setminus I$ are monadic; this case corresponds to structures endowed with a colouring of its vertices).  Consequently, in this case the generic structure $\relsys{U}'$ has not only the universal shadow $\relsys{U}$ but in fact this shadow $\relsys{U}$ has a finite core $\relsys{D}$ which is thus a finite hom-universal structure and thus a dual. This is stated as Corollary \ref{nasdual}. In this way we reprove the finite duality theorem \cite{Nesetril2000}.  This is particularly pleasing: (in this sense) duals are generic objects and this indicates yet another context of hom-dualities.

In Section \ref{urysohnsection} we include further corollaries of our method to metric spaces (which can be treated analogously to forbidding cycles). In Section \ref{bounded} we solve a related extremal problem: we show that our method lead to the optimal arities of added lifts. This is based on a non-trivial Ramsey-type result (Lemma \ref{sets}).
The results of Section~\ref{bounded} may be formulated by means of relational complexity \cite{Hartman2014}.
This we shall elaborate in a separate paper with David Hartman.

\section{Preliminaries}

Let us add a few more definitions. For a structure $\relsys{A}=(A,\rel{A}{i},i\in I)$ the {\em Gaifman graph}
(in combinatorics often called {\em 2-section}) is the graph $\relsys{G}_\relsys{A}$ with vertices
$A$ and all those edges which are a subset of a tuple of a relation of
$\relsys{A}$: $$\relsys{G}_\relsys{A}=(V,E)$$ where $\{x,y\}\in E$ if and only if $x\neq y$ and there exists
tuple $\vec{v}\in \rel{A}{i},i\in I$ such that $x,y\in \vec{v}$.

A {\em g-cut} in $\relsys{A}$ is a subset $C$ of $A$ such that the Gaifman
graph $G_\relsys{A}$ is disconnected by removing set $C$. That is, there are
vertices $u,v\in A\setminus C$ that belong to the same connected component of
$G_\relsys{A}$ but to different connected components of
$G_\relsys{A}\setminus C$.
A {\em cut} in $\relsys{A}$ is subset $C$ of $A$
such that there are vertices $u,v\in A\setminus C$ that belong to the same
connected component of $\relsys{A}$ but to different connected components of $\relsys{A}\setminus C$.
(Observe that not every cut is a g-cut. With relations of arity greater than 2,
$G_{\relsys{A}\setminus C}$ may be different from $G_{\relsys{A}}\setminus C$.)

For g-cut $C$ in relational structure $\relsys{A}$  {\em g-component of
$\relsys{A}$ with g-cut $C$} is a structure $\relsys{A}_1$ induced by
$\relsys{A}$ on some connected component of $G_\relsys{A}\setminus C$.

We will make use of the following simple observation about the neighbourhood and
g-components.
\begin{observation}
\label{cuty}
Let $\relsys{A}_1$ be a g-component of $\relsys{A}$ with g-cut $C$. Then the neighbourhood $N_\relsys{A}(A_1)$ is a subset of $C$.
Moreover $N_\relsys{A}(A_1)$ is a g-cut and $\relsys{A}_1$ is one of the g-components of $\relsys{A}$ with g-cut $N_\relsys{A}(A_1)$.
\end{observation}

Given a structure $\relsys{A}$ with g-cut $C$ and two (induced) substructures
$\relsys{A}_1$ and $\relsys{A}_2$, we say that $C$ {\em g-separates}
$\relsys{A}_1$ and $\relsys{A}_2$ if there are g-components $\relsys{A}'_1\neq
\relsys{A}'_2$ of $\relsys{A}$ with g-cut $C$ such that $A_1\subseteq A'_1$ 
and $A_2\subseteq A'_2$.

\begin{defn}
\label{def:separating}
Let $C$ be a g-cut in structure $\relsys{A}$. Let $\relsys{A}_1\neq \relsys{A}_2$
be two g-components of $\relsys{A}$ with g-cut $C$.  We call $C$ {\em
minimal g-separating g-cut} for $\relsys{A}_1$ and $\relsys{A}_2$ in $\relsys{A}$ if
$C=N_\relsys{A}(A_1)=N_\relsys{A}(A_2)$.
\end{defn}
For brevity, we can omit one or both $g$-components when speaking about
a minimal g-separating g-cut. Explicitely,
we call g-cut $C$ {\em minimal g-separating} for $\relsys{A}_1$
in $\relsys{A}$ if there exists another structure $\relsys{B}$ such that $C$
is minimal g-separating for $\relsys{A}_1$ and $\relsys{B}$ in $\relsys{A}$.  A g-cut
$C$ is {\em minimal g-separating} in $\relsys{A}$ if there exists structures
$\relsys{B}_1$ and $\relsys{B}_2$ such that $C$ is minimal g-separating for
$\relsys{B}_1$ and $\relsys{B}_2$ in $\relsys{A}$.

The name of minimal g-separating g-cut is justified by the following (probably folkloristic) proposition.
\begin{prop}
Let $\relsys{A}$ be a connected relational structure, $C$ a g-cut in $\relsys{A}$ and
$\relsys{A}_1$ and $\relsys{A}_2$ (induced) connected substructures of $\relsys{A}$
g-separated by $C$. Then there exists a minimal g-separating g-cut $C'\subseteq C$ that
g-separates $\relsys{A}_1$ and $\relsys{A}_2$ in $\relsys{A}$.
Moreover if $N_\relsys{A}(A_1)\subseteq C$ (or, equivalently, $\relsys{A}_1$ is a g-component of $\relsys{A}$ with g-cut $C$), then $C'\subseteq N_\relsys{A}(A_1)$.
\label{prop:sep}
\end{prop}
\begin{proof}
\begin{figure}
\centerline{\includegraphics{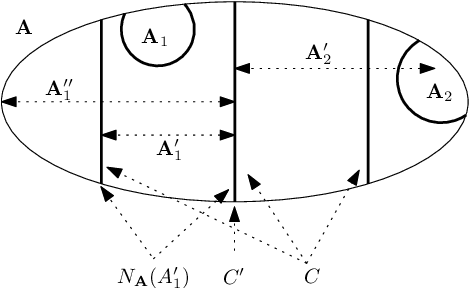}}
\caption{Construction of a minimal g-separating g-cut $C'$ separating $\relsys{A}_1$ and $\relsys{A}_2$ in $\relsys{A}$.}
\label{constructioncut}
\end{figure}
We will construct a series of g-cuts and g-components as depicted in Figure~\ref{constructioncut}.

Denote by $\relsys{A}'_1$ the g-component of $\relsys{A}$ with g-cut $C$
containing $\relsys{A}_1$ (and thus not containing $\relsys{A}_2$).  By Observation~\ref{cuty},
$N_\relsys{A}(A'_1)\subseteq C$ is g-cut that g-separates $\relsys{A}'_1$ and
$\relsys{A}_2$ (because $\relsys{A}'_1$ is also g-component of $\relsys{A}$ with g-cut
$N_\relsys{A}(A'_1)$ and $\relsys{A}'_1$ do not contain $\relsys{A}_2$).  

Now
consider g-component $\relsys{A}'_2$ of $\relsys{A}$ with g-cut
$N_\relsys{A}(A'_1)$ containing $\relsys{A}_2$.  Put $C'=N_\relsys{A}(A'_2)$. By
Observation~\ref{cuty}, $C'\subseteq N_\relsys{A}(A'_1)\subseteq C$ is g-cut and
$\relsys{A}'_2$ (not containing $\relsys{A}_1$) is one of its g-components.

Denote by $\relsys{A}''_1$ the g-component of $\relsys{A}$ with g-cut $C'$
containing $\relsys{A}_1$. It follows that $C'$ g-separates $\relsys{A}''_1$ (that contains $\relsys{A}_1$) and $\relsys{A}'_2$ (that contains $\relsys{A}_2$).

To see that $C'$ is minimal g-separating for $\relsys{A}''_1$ 
and $\relsys{A}'_2$ it remains
to show that every vertex in $C'=N_\relsys{A}(A'_2)$ is also in $N_\relsys{A}(A''_1)$.  This is true
because every vertex of $C'$ is in $N_\relsys{A}(A'_1)$ and $\relsys{A}'_1$ is
substructure of $\relsys{A}''_1$. 

\end{proof}
Observe that every inclusion minimal g-cut is also minimal g-separating, but
not vice versa. Every minimal g-separating g-cut $C'\subset C$ that g-separates
$\relsys{A}_1$ and $\relsys{A}_2$ is however also inclusion minimal g-cut that
separates $\relsys{A}_1$ and $\relsys{A}_2$.  

\section {Basic construction ($\Forb(\relsys{\F})$ classes)}
\label{sec:basis}

If $C$ is a set of vertices then $\overrightarrow{C}$ will denote a tuple (of
length $|C|$) formed by all the elements of $C$. Alternatively, $\overrightarrow{C}$
is an arbitrary linear ordering of $C$.
A {\em rooted structure} $\Piece$ is a pair $(\relsys{P},\overrightarrow{R})$
where $\relsys{P}$ is a relational structure and $\overrightarrow{R}$ is an
tuple consisting of distinct vertices of $\relsys{P}$. $\overrightarrow{R}$ is
called the {\em root} of $\Piece$ and the size of $\overrightarrow{R}$ is the
{\em width} of $\Piece$.  We say that rooted structures
$\Piece_1=(\relsys{P}_1,\overrightarrow{R}_1)$ and
$\Piece_2=(\relsys{P}_2,\overrightarrow{R}_2)$ are {\em isomorphic} if there is
a function $f:P_1\to P_2$ that is an isomorphism of structures $\relsys{P}_1$
and $\relsys{P}_2$ and $f$ restricted to $\overrightarrow{R}_1$ is a monotone
bijection between $\overrightarrow{R}_1$ and $\overrightarrow{R}_2$ (we denote
this $f(\overrightarrow{R}_1)=\overrightarrow{R}_2$).

The following is our basic notion which
resembles decomposition techniques standard in graph theory and thus we adopted similar terminology.

\begin{defn}
Let $\relsys{A}$ be a connected relational structure and $R$ a minimal g-separating
g-cut for g-component $\relsys{C}$ in $\relsys{A}$. A {\em piece} of a relational structure
$\relsys A$ is then a rooted structure $\Piece=(\relsys{P},\overrightarrow{R})$, where the
tuple $\overrightarrow{R}$ consists of the vertices of the g-cut $R$ in a (fixed)
linear order and $\relsys {P}$ is a structure induced by $\relsys{A}$ on $C\cup R$.
\end{defn}

Note that since $\relsys{P}$ is the union of a g-component and its neighbourhood it follows that
the pieces of a connected structure are always connected structures.

All pieces are considered as rooted structures: a piece $\Piece$ is a structure $\relsys{P}$ rooted at $\overrightarrow{R}$. Accordingly, we say
 that pieces $\Piece_1=(\relsys{P}_1,\overrightarrow{R}_1)$ and
$\Piece_2=(\relsys{P}_2,\overrightarrow{R}_2)$ are {\em isomorphic} if there is function $\varphi:P_1\to P_2$ that is isomorphism of structures $\relsys{P}_1$ and $\relsys{P}_2$ and $\varphi$ restricted to $\overrightarrow{R}_1$ is the monotone bijection between $\overrightarrow{R}_1$ and $\overrightarrow{R}_2$ (we denote this
$\varphi(\overrightarrow{R}_1)=\overrightarrow{R}_2$).

\begin{figure}
\centerline{\includegraphics{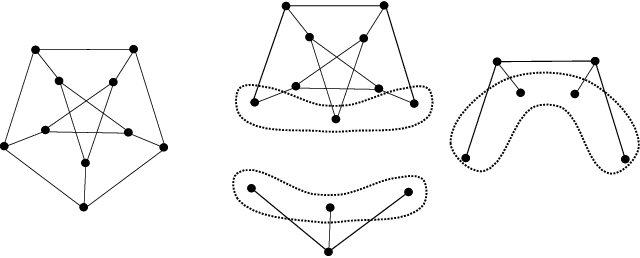}}
\caption{Pieces of Petersen graph up to isomorphisms (and a permutations of roots). }
\label{Petersoni}
\end{figure}
Observe that for relational trees, pieces are equivalent to rooted  branches. Figure \ref{Petersoni} shows all pieces of the Petersen graph.

\begin{lem}
\label{lem:ama}
Let $\Piece_1=(\relsys{P}_1,\overrightarrow{R}_1)$ be a piece of structure $\relsys A$ and $\Piece_2=(\relsys{P}_2,\overrightarrow{R}_2)$ a piece of $\relsys{P}_1$.  If $R_1\cap P_2\subseteq R_2$, then $\Piece_2$ is also a  piece of $\relsys A$.
\label{kousky}
\end{lem}
\begin{proof}
Denote by $\relsys{C}_1$ connected component of $\relsys{A}\setminus R_1$ that produces $\Piece_1$ (i.e. $C_1\cup R_1=P_1$).  Denote by $\relsys{C}_2$ component of $\relsys{P}_1\setminus R_2$ that produces $\Piece_2$ (i.e. $C_2\cup R_2=P_2$). 
As $R_1\cap P_2\subseteq R_2$ one can check that then $\relsys{C}_2$ is contained in $\relsys{C}_1$ and every vertex of $\relsys{A}$ connected by tuple to any vertex of $\relsys{C}_2$ is
contained in $\relsys{P}_1$. Thus $\relsys{C}_2$ is also connected component of
$\relsys{A}$ created after removing vertices of $R_2$.
\end{proof}

Let $\F$ be a fixed finite set of connected connected finite relational structures of (finite) type
$\Delta$. For construction of universal structure of $\Forb(\F)$ we use special
lifts, called $\F$-lifts.  

Towards this end let $\Piece_i=(\relsys{P}_i,\overrightarrow{R}_i), i\in I'$ be an enumeration of all pieces of all relational.
structures $\relsys{F}\in \F$. Notice that there are only finitely many pieces. This enumeration of pieces is fixed throughout this section.

Relational structure $\relsys X=(\relsys{A}, (\ext{X}{i},i\in I'))$ is
called {\em $\F$-lift} of relational structure $\relsys A$ when the arities of
relations $\ext{X}{i},i\in I'$, correspond to
$|\overrightarrow{R}_i|$ (i.e. the width of $\Piece_i$).

For relational structure $\relsys A$ we define {\em canonical lift $\relsys{X}=L({\relsys A})$}
by putting $(v_1,v_2,\ldots,v_l)\in \ext{X}{i}$ if and only if there is homomorphism
$\varphi$ from $\relsys{P}_i$ to $A$ such that $\varphi(\overrightarrow
{R}_i)=(v_1,v_2,\ldots,v_l)$.

We will use the following notion of complete lifts:
\begin{defn}
Canonical lift $L({\relsys A})$ is {\em complete} on $B\subseteq A$ if for every  $\relsys{C}\in \Forb(\F)$
such that $\relsys{C}$ contains $\relsys{A}$ as induced substructure, the lift induced on $B$ by $L(\relsys{A})$ is
the same as the lift induced by $L(\relsys{C})$.
We say that lift $\relsys{X}$ is {\em complete} if there exists $\relsys{A}\in
\Forb(\F)$ such that $\relsys{X}$ is induced on $X$ by $L(\relsys{A})$ and the
canonical lift $L(\relsys{A})$ is complete on $X$.
\end{defn}

Denote by $\Lifts$ the class of all complete lifts.

For $\relsys{X}\in \Lifts$ we denote by $W(\relsys{X})$ a structure $\relsys{A}\in\Forb(\F)$ such that structure $\relsys{X}$ is induced on $X$ by $L(\relsys{A})$ and $L(\relsys{A})$ is complete on $X$. $W(\relsys{X})$ is called a {\em witness} of the fact that $\relsys{X}$ belongs to $\Lifts$.

We shall remark that the definition of $\Lifts$ and witness is different from
the definition used by \cite{Hubicka2013} --- here we add the notion of completeness.
This simplifies proof of the amalgamation property.
In fact the completeness imply that the resulting universal
structure is existentially complete (while ones constructed in \cite{Hubicka2013}
are not). Moreover when we speak about finite families only, the techniques can be easily
extended to regular families of forbidden substructures in the sense if \cite{Hubicka2013}.

The key technical part of our construction is expressed by the following:
\begin{lem}
Let $\relsys{A}$ and $\relsys{B}$ be both witnesses of $\relsys{X}$. Then the free amalgam of $\relsys{A}$ and $\relsys{B}$ over $X$ is also a witness.
\end{lem}
\begin{proof}
Denote by $\relsys{D}$ the free amalgam of $\relsys{A}$ and $\relsys{B}$ over $X$.
From the completeness of $\relsys{X}$ in both $\relsys{A}$ and $\relsys{B}$ we know that $\relsys{D}$ is a witness of $\relsys{X}$ if
$\relsys{D}\in \Forb(\F)$. Assume, to the contrary, that
there is an $\relsys{F}\in \F$ and a homomorphism $f$ from $\relsys{F}$ to $\relsys{D}$.
Homomorphism $f$ partitions the vertex set of $\relsys{F}$ into three sets defined as follows: $F_X$
are vertices with image in $X$, $F_A$ are vertices with image in
$A\setminus X$ and $F_B$ are vertices with image in $B\setminus X$.
Without loss of generality we can assume that $\relsys{F}$ and $f$ was chosen in
a way so $|F_A|$ is minimal.

Observe that $F_X$ is a cut of $\relsys{F}$ separating $F_A$ and $F_B$. 
 Denote by $\Piece_i$ a piece with root contained in $F_X$ containing a vertex of $F_A$ (such piece can be obtained by Proposition~\ref{prop:sep}).
If $f(\overrightarrow{R}_i)\in \ext{X}{i}$ (and thus also
$f(\overrightarrow{R}_i)\in \ext{L(A)}{i}$ and $f(\overrightarrow{R}_i)\in
\ext{L(B)}{i}$) then there exists a homomorphism $f':\Piece_i\to \relsys{B}$
such that $f'(\overrightarrow{R}_i)= f(\overrightarrow{R}_i)$. Consider function
$f'':F\to D$ defined as follows:
\begin{enumerate}
\item $f''(x)=f'(x)$ for $x\in P_i$,
\item $f''(x)=f(x)$ otherwise.
\end{enumerate}
$f''$ is a homomorphism $\relsys{F}\to \relsys{D}$ that uses fewer vertices of $F_A$ and possibly more vertices of $F_X\cup F_B$.  We call this an {\em flip operation}. When piece has root in $F_X$, flip operation moves the image of a piece from one part of the amalgam to the other.

By minimality of $F_A$ we thus know that $f(\overrightarrow{R}_j)\notin \rel{X}{j}$.
If $P_j\subseteq A\cup X$, then by the definition of canonical lift we have $f(\overrightarrow{R}_j)\in
\rel{L(A)}{j}$ a contradiction.
We thus conclude that every piece with root in $F_X$ containing a vertex of
$F_A$ must also contain a vertex of $F_B$.

Choose $\Piece_j$ to be piece containing both vertices of $F_A$ and $F_B$ with minimal number of non-root vertices among pieces with this property.
If $\Piece_j$ contains a sub-pieces with root in $F_X$ contained in $F_X\cup F_B$, we can perform the flip operation, this time replacing
vertices with images in $F_B$ by vertices with image in $F_X\cup F_A$.
If this procedure eliminates all vertices of $P_j\cap F_B$ we get
a homomorphism $f':\relsys{P}_j\to \relsys{A}$, $f'(\overrightarrow{R}_j)=f(\overrightarrow{R}_j)$, and therefore $f(\overrightarrow{R}_j)\in \rel{L(A)}{j}$ that contradicts
minimality of $|F_A|$.

Denote by $\relsys{A}_{j}$ a g-component of $\relsys{F}$ with g-cut $F_X$ contained in $\Piece_{j}$ consisting of vertices of $F_A$ and by $\relsys{B}_{j}$ a g-component of $\relsys{F}$ with g-cut $F_X$ contained in $\Piece_{j}$ consisting of vertices of $F_B$ that can not be eliminated from $\Piece_{j}$ by the flip operations.
Denote by $F'$ a set vertices of any connected component of $\relsys{F}\setminus P_j$ such that $R_j\subseteq N(F')$ (such component exists because $R_j$ is a minimal separating cut).
By application of Proposition~\ref{prop:sep} on g-cut $F_X\cap P_{j}$ with $F'$
and $B_j$, one gets that $R_{j}\subseteq N(B_{j})$. Otherwise one would obtain
a sub-piece that would contradict minimality of $\Piece_{j}$ or an assumption that $\relsys{B}_{j}$ can not be eliminated (and
thus it is not a contained in a piece consisting only of vertices $F_X\cup F_B$). The symmetric argument
gives $R_{j}\subseteq N(A_{j})$. Now by application of Proposition~\ref{prop:sep} with cut $F_X$ and
components $\relsys{A}_{j}$ and $\relsys{B}_{j}$ we obtain minimal separating g-cut $C$. Clearly $R_j\subseteq C$ because $R_j\subseteq N(A_j)\cap N(B_j)$. $C$ must contain some additional vertices of $F_X\cup (P_j\setminus R_j)$ because $\relsys{P}_j\setminus R_j$ is connected and $F_X$ separates $A_j$ and $B_j$. The pieces obtained are thus a proper sub-pieces
of $\Piece_j$ that either contain both vertices of $F_A$ and $F_B$ or they can be used for the flip operations. 
In all these cases this yields a contradiction.
\end{proof}
Now we can prove Theorem~\ref{univthm} formulated more precisely as:
\begin{thm}
\label{mainthm}
Let $\F$ be a finite set of finite connected relational structures, then the class
$\Lifts$ is an amalgamation class.
Consequently, there is a generic structure $\relsys{U}$ in $\Lifts$ and its shadow $\sh(\relsys{U})$ is a universal structure for class $\Forb(\F)$.
Moreover $\sh(\relsys{U})$ is model complete and $\Aut(\relsys{U})=\Aut(\sh(\relsys{U})))$.
\label{mainth}
\end{thm}

\begin{proof}
The class $\Lifts$ is clearly hereditary, isomorphism closed and has the joint
embedding property. $\Lifts$ is countable, because there are only countably
many structures in $\Forb(\F)$ (because type $\Delta$ is finite) and thus also
countably many lifts. Thus to show that $\Lifts$ is an amalgamation class it remains to
verify that $\Lifts$ has the amalgamation property.

Consider $\relsys{X},\relsys{Y},\relsys{Z}\in \Lifts$. Assume that structure
$\relsys{Z}$ is substructure induced by both $\relsys{X}$ and
$\relsys{Y}$ on $Z$ and without loss of generality assume that $X\cap Y=Z$.

Put
$$\relsys{A}=W(\relsys{X}),$$
$$\relsys{B}=W(\relsys{Y}),$$
$$\relsys{C}=\sh(\relsys{Z}).$$

Now consider $\relsys{D}$, the free amalgam of $\relsys{A}$ and $\relsys{B}$
over $\relsys{C}$.  By Lemma~\ref{lem:ama}, $\relsys{D}$ is a witness of
$\relsys{Z}$ and also a witness of $\relsys{A}$ and $\relsys{B}$.  Now find
$\relsys{E}\in \Forb(\F)$ containing $\relsys{D}$ as an substructure such that
$L(\relsys{E})$ is complete on $D$.
It follows structure induced on $D$ on $L(\relsys{E})$ is the amalgamation of $\relsys{X}$ and $\relsys{Y}$ over $\relsys{Z}$.

We show the model-completeness of $\sh(\relsys{U})$ as follows: 
Recall that a structure is model complete if every formula is equivalent to an existential formula.
 We use the fact that $\relsys{U}$ is model complete (because every ultrahomogeneous structure is model complete).
 Given a formula $\phi$ in the language of $\sh(\relsys{U})$,
by a model completeness of $\relsys{U}$ one can obtain an equivalent existential formula $\phi'$ in the language of $\relsys{U}$.
We translate this formula back to language of $\sh(\relsys{U})$ by replacing every $\vec{v}\in \ext{U}{i}$ by
an existential formula testing the presence of a rooted homomorphism $f$ from $\Piece_i\to \sh(\relsys{U})$ such that
$f(\overrightarrow{R}_i)=\vec{v}$.
Now observe that by completeness of witnesses one can also turn $\vec{v}\notin \ext{U}{i}$ into an existential
formula describing all cases where a homomorphism from $\Piece_i \to \sh(\relsys{U})$ with $f(\overrightarrow{R}_i)=\vec{v}$
would produce a homomorphic image of some $\relsys{F}\in \F$. In other words, for every piece $\Piece_i$ there is a set
of rooted structures $\overline{\Piece}_i^1$, $\overline{\Piece}_i^2$, \ldots,$\overline{\Piece}_i^n$
such that $\vec{v}\notin \ext{U}{i}$ imply existence of $m$, $1\leq m\leq n$ and a rooted homomorphism
$g:\overline{\Piece}_i^m\to \sh(\relsys{U})$ such that the root is mapped to $\vec{v}$ and moreover the free
amalgam of $\overline{\Piece}_i^m$ and $\Piece_i$ is not in $\Forb(\F)$.
The same observations also give that automorphism groups of the lift and the shadow are equivalent.
\end{proof}

\section{Bounding arities}
\label{bounded}

The expressive power of lifts can be limited in several ways. For example, it
is natural to restrict arities of the newly added relations.  It follows from
the above proof that the arities of new relations in our lifted amalgamation
class $\Lifts$ depend on the maximal size of an minimal separating g-cut of the
forbidden structure.

In this section we completely characterise the minimal arity of
ultrahomogeneous lifts of classes $\Forb(\F)$ where $\F$ is a finite family of
connected structures.  This involves a non-trivial Ramsey type statement stated
bellow as Lemma \ref{sets}. As a warm up, we first show that generic universal
graph for the class $\Forb(C_5)$ can not be constructed by finite monadic
lifts:

Consider, for contradiction, that there exists a monadic lift $\relsys{U}'$ which is both a generic relational structure and whose
shadow $\relsys{U}$ is universal for the class $\Forb(C_5)$. Since all extended
relations are monadic, we can view them as a finite colouring of vertices.
For $v\in \relsys{U}$ we will denote by $c(v)$ the colour of $v$ or, equivalently, set of
all extended relations $\ext{U}{i}$ such that $(v)\in \ext{U}{i}$.

Since graphs in $\Forb(C_5)$ have unbounded chromatic number, we know that
chromatic number of $\relsys{U}$ is infinite. Consider decomposition of
$\relsys{U}$ implied by $c$. Since the range of $c$ is finite, one of the
graphs in this decomposition has infinite chromatic number. Denote this subgraph
by $\relsys{S}$.

In fact is suffices that $\relsys{S}$ is not bipartite. Thus $\relsys{S}$ contains an odd circle. The shortest odd cycle has length $\geq 7$ and thus $\relsys{S}$ contains
 induced
path of length 3 formed by vertices $p_1,p_2,p_3,p_4$.  Additionally there is a vertex $v$ of degree at least 2.  Because the graph is triangle free, the vertices $v_1$ and $v_2$ connected to $v$ are not connected by an edge.

From genericity of $\relsys{U}$ we know that the partial isomorphism mapping
$v_1\to p_1$ and $v_2\to p_4$ can be extended to an automorphism $\varphi$ of $\relsys{U}$.
The vertex $\varphi(v)$ is connected to $p_1$ and $p_4$ and thus together with
$p_1,p_2,p_4$ contains either triangle or $5$-cycle.  It follows that generic $\relsys{U}$ cannot be defined by monadic lifts of $\Forb(\F)$.

In this section we prove that there is nothing special here about arity 2 and about the pentagon. We can determine the minimal arity of the ultrahomogeneous lifts for general classes $\Forb(\F)$. Towards this end we shall need a Ramsey type statement which we formulate after introducing the following:

Let $S$ be a finite set with partition $S_1\cup S_2\cup\ldots\cup S_n$. For $v\in S$ we denote by $i(v)$ the index $i$ such that $v\in S_i$. Similarly, for a tuple $\vec{x}=(x_1,x_2,\ldots,x_t)$ of elements of $S$ we denote by $i(\vec{x})$ the tuple $(i(x_1), i(x_2),\ldots,i(x_t))$. A tuple $\vec{x}$ is called {\em transversal} if all indices $i(x_j)$ are distinct. We make use of the following:
\begin{lem}
\label{sets}
For every $n\geq 2$, $1\leq r<n$ and $K$ integers there is relational structure $\relsys{S}=(S,\relS)$, with vertices $S=S_1\cup S_2\cup\ldots\cup S_n$ (sets $S_i$ are mutually disjoint) and single relation $\relS$ of arity $2n$ with the following properties:
\begin{enumerate}
  \item Every $2n$-tuple $(v_1,u_1,v_2,u_2,\ldots, v_n,u_n)\in \relS$ is in the following form:
  \begin{enumerate}
   \item $v_1$, $v_2$, \ldots, $v_n$ and $u_1$, $u_2$, \ldots, $u_n$ transversal,
   \item $v_i\neq u_i$ $i=1,2,\ldots, n$.
  \end{enumerate}
  \item For every $\vec{v},\vec{u}\in \relS$, $\vec{v}\neq\vec{u}$, $\vec{v}$ and $\vec{u}$ have at most $r$ common vertices.
  \item 
For every colouring of transversal $r$-tuples on $\relsys{S}$ using $K$ colours there is $2n$-tuple $(v_1,u_1,v_2,u_2,\ldots, v_n,u_n)\in \relS$ such that for any $1\leq i_1<i2<\ldots<i_r\leq n$ all the transversal r-tuples $\vec{x}$ with $i(\vec{x})=(i_1,i_2,\ldots,i_r)$ have the same colour.
\end{enumerate}
\end{lem}
\begin{proof}
This statement follows from results obtained by Ne\v set\v ril and R\"odl
\cite{Nevsetvril1981}. Although not stated explicitely, this is ``partite version''
of the main result of \cite{Nevsetvril1981}. It can be also obtained directly by
means of amalgamation method, see \cite{Nevsetvril1995,Nevsetvril1987}. In this paper this result plays
auxiliary role only and we omit the proof.
\end{proof}
Let $\relsys{S}=(S,\relS)$ be a relational structure with relation $\relS$ of arity $2n$. Given a rooted relational structure $(\relsys{A}, \overrightarrow{R})$ of type $\Delta$ with $\overrightarrow{R}=(r_1,r'_1,r_2,r'_2,\ldots, r_n,r'_n)$ we denote by $\relsys{S}*(\relsys{A},\overrightarrow{R})$ the following relational structure $\relsys{B}$ of type $\Delta$:

\begin{enumerate}
\item 
The vertices of $\relsys{B}$ are equivalence classes of equivalence $\sim$ on $\relS \times A$
generated by the following pairs:

$$(\vec{v}, r_i)\sim (\vec{u}, r_i) \hbox{ iff } \vec{v}_{2i}=\vec{u}_{2i},$$
$$(\vec{v}, r'_i)\sim (\vec{u}, r'_i) \hbox{ iff } \vec{v}_{2i+1}=\vec{u}_{2i+1},$$
$$(\vec{v}, r_i)\sim (\vec{u}, r'_i) \hbox{ iff } \vec{v}_{2i}=\vec{u}_{2i+1}.$$

\item
Denote by $[\vec{v}, r_i]$ the equivalence class of $\sim$ containing $(\vec{v}, r_i)$.
We put $\vec{v}\in \rel{B}{j}$ if and only of $\vec{v}= ([\vec{u},v_1],[\vec{u},v_2],\ldots, [\vec{u},v_t])$ for some $\vec{u}\in \relS$ and $(v_1,v_2,\ldots, v_t)\in \rel{A}{j}$.
\end{enumerate}

This construction is commonly used in the graph homomorphism context as {\em indicator construction}, see e.g. \cite{Hell2004}. It essentially means replacing every tuple of $\relS$ by disjoint copy of $\relsys{A}$ with roots $\overrightarrow{R}$ identified with vertices of the tuple.

For given vertex $v$ of $\relsys{S}*(\relsys{A},\overrightarrow{R})$ such that $v=[\vec{u}, r_i]$ (or $v=[\vec{u}, r'_i]$) we will call vertex $v'=\vec{u}_{2i}$ (or $v'=\vec{u}_{2i+1}$ respectively) {\em the vertex corresponding to $v$ in $\relsys{S}$}.  Note that this gives correspondence between vertices of $\relsys{S}$ and those vertices of $\relsys{S}*(\relsys{A},\overrightarrow{R})$ restricted to vertices $[\vec{v}, r_i]$ and $[\vec{v}, r'_i]$.

Finite family finite relational structures is called {\em minimal} if and only if there
all structures in $\F$ are cores and there is no homomorphism in between two
structures
in $\F$.

The following is the main result of this section.

\begin{thm}
\label{aritythm}
Denote by $\F$ minimal family of finite connected relational structures. There is lift of class $\Age(\Forb(\F))$ which is an amalgamation class and which contains new relations of arity at most $r$ if and only if all minimal g-cuts of $\relsys{F}\in \F$ consist of at most $r$ vertices.
\end{thm}
\begin{proof}
One direction of this follows from the construction of lift $\Lifts$ given above in the proof of Theorem~$\ref{mainthm}$. This construction
 adds
relations of arities corresponding to the sizes of minimal g-cuts of
$\relsys{F}\in \F$ and thus the sufficiency of Theorem~\ref{aritythm} follows directly from the proof
of Theorem $\ref{mainthm}$.

In the opposite direction fix class $\F$, $r\geq 1$ such that all relational structures
in $\F$ contains minimal separating g-cuts $C=\{r_1,r_2,\ldots, r_n\}$ of size 
at most $n$. Moreover assume that $\relsys{F}\in \F$ has a minimal separating g-cut $C$ of size $n$.  Assume, for contradiction, that there exists lift $\K$ of class
$\Forb(\F)$ such that $\K$ is an amalgamation class and contains new relations
of arities at most $r$. Denote by $K$ the number of different relational structures 
on the set $\{1,2,\ldots, r\}$ which appear in $\K$.
As $C$ is a minimal separating g-cut (see the Definition~\ref{def:separating}) the components of $\relsys{F}\setminus C$ can be enumerated as $\relsys{B}_1$, $\relsys{B_2}$, $\relsys{B}_3$,\ldots, $\relsys{B}_t$ such that
\begin{enumerate}
 \item $C$ is the neighbourhood of $\relsys{B}_1$,
 \item $C$ is the neighbourhood of $\relsys{B}_2$,
 \item neighbourhood of $B_i$, $i>2$ is contained in $\relsys{C}$.
\end{enumerate}
Let $\relsys{F}_0$ be the substructure of $\relsys{F}$ induced by $\overrightarrow{C}=(c_1,c_2,\ldots,c_n)$.

Consider the pieces $\Piece_1=(\relsys{P}_1,\overrightarrow{C}_1)$, $\Piece_2=(\relsys{P}_2,\overrightarrow{C}_2)$ generated by $\overrightarrow{C}$ and
$\relsys{B}_1$ and $\relsys{B}_2$. We assume that $\overrightarrow{C}_1$ and $\overrightarrow{C}_2$ are disjoint sets:
$\overrightarrow{C}_1=(c^1_1,c^2_1,\ldots, c^n_1)$ and
$\overrightarrow{C}_2=(c^1_2,c^2_2,\ldots, c^n_2)$.

Denote also $\Piece=(\relsys{F}\setminus \relsys{B}_1, \overrightarrow{C}_2)$ the rooted structure obtained by deleting $\relsys{B}_1$ form $\relsys{F}$.
Observe that the amalgamation of $\Piece_1$ and $\Piece$ along the roots is forbidden ($\Piece$ is not a piece). On the other hand both $\Piece_1$ and $\Piece_2$
are pieces (however their amalgamation need not to be forbidden).

\begin{figure}
\label{Hconstruct}
\centerline{\includegraphics{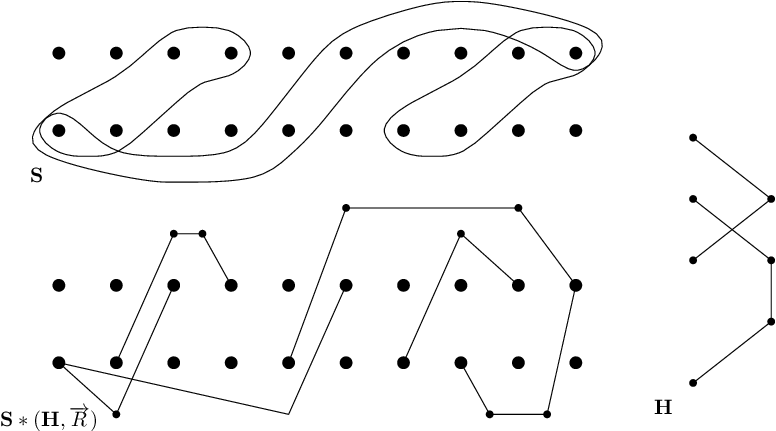}}
\caption{The construction of $\relsys{D}$.}
\end{figure}

Let $H=(H,\overrightarrow{R})$ be disjoint union of $\Piece_1$ and $\Piece_2$ where the root vertices are denoted as $\overrightarrow{R}=(c^1_1, c^2_1, c^1_2, c^2_2,\ldots, c^1_n,c^2_n)$.
In this situation let $\relsys{S}$ be the relational system of arity $2n$ from Lemma~\ref{sets} and we put $\relsys{D}=\relsys{S}*(\relsys{H},\overrightarrow{R})$. For two pieces of $5$-cycle this construction is shown at Figure~\ref{Hconstruct}.

As the family $\F$ is minimal it follows that $\relsys{H}\in \Forb(\F)$. We prove that even $\relsys{D}=\relsys{S}*(\relsys{H},\overrightarrow{R})$ belongs to $\Forb(\F)$. Assume, to contrary, that there exists $\relsys{F}'\in \F$ and a homomorphism $f:\relsys{F}'\to \relsys{D}$. The relational system $\relsys{S}$ homomorphically maps to a relational structure consisting of single tuple $(\{C_1,\ldots,C_n\},\{(C_1,\ldots,C_n)\})$ (by mapping $x\to i$ for $x\in S_i$) and thus $\relsys{S}*(\relsys{H},\overrightarrow{R})$ homomorphically maps to structure $\relsys{F}$ (by mapping $g:[(\overrightarrow{r},x)]\to x$).

It follows that $g\circ f$ is a homomorphism $\relsys{F}'\to \relsys{F}$ and as $\F$ is a minimal family we get that $\relsys{F}'=\relsys{F}$ and, moreover, $g\circ f$ is an automorphism of $\relsys{F}$. It also follows that $f$ is an embedding of $\relsys{F}$ into $\relsys{D}$, see Figure~\ref{fig:maps}.
\begin{figure}
\centerline{\xymatrix{\relsys{F} \ar[r]^-f \ar[rd]^h & \relsys{S}*(\relsys{H},\overrightarrow{R}) \ar[d]^g\\ & \relsys{F}}}
\caption{Embeddings between $\relsys{F}$ and $\relsys{D}$.}
\label{fig:maps}
\end{figure}

Now $g([(\overrightarrow{R},c^j_i)])=c_i$ for $j=1,2$, $i=1,\ldots, n$, and $g([(\overrightarrow{R},x)])=x$ for all $x\in \relsys{H}\setminus R$.

Consider the isomorphism $h=(g\circ f)^1:\relsys{F}\to \relsys{F}$. Put $\overline{C}$, $\overline{B}_1$, $\overline{B}_2$, \ldots,$\overline{B}_t$, $\overline{\Piece}$, $\overline{\Piece}_1$, $\overline{\Piece}_2$ the isomorphic copies of $C$, $B_1$, $B_2$, \ldots, $B_t$, $\Piece$, $\Piece_1$ and $\Piece_2$.
We have that for every $i$ there are $\vec{u},\vec{v}$ such that $f(h(c_i))\in\{[(\vec{u},c^i_1)],[(\vec{v},c^i_2)]\}$. However $\overline{B}_1$ and $\overline{B}_2$ are connected and have $\overline{C}$  as their neighbourhood and thus $f(h(c_i))=[(\vec{u},c^i_1)]$ for some $\vec{u}$ and every $i$ or $f(h(c_i))=[(\vec{u},c^2_1)]$ for some $\vec{u}\in \relS$ and every $i$.
This follows from definition of $D=\relsys{S}*(\relsys{H},\overrightarrow{R})$: $D\setminus S$ is disconnected graph formed by copies of $B_1, B_2, \ldots, B_t$.
But applying again the connectivity and the definition of $\relsys{S}*(\relsys{H},\overrightarrow{R})$ there exists $\vec{u}$ such that $f(x)=[(\vec{u},x)]$ as distinct $\vec{u}$ intersect in at most $r$-points.
However this is a contradiction as then $f$ cannot be an embedding. It follows that $\relsys{D}\in \Forb(\F)$.

Take generic lift $\relsys{U}\in \K$ which uses at most $r$-tuples.  Every embedding $\Phi:\relsys{D}\to \sh(\relsys{U})$ ($\sh(\relsys{U})$ is shadow of $\relsys{U}$) imply $K$ colouring of $r$-tuples of $D$ (colours are defined by the additional relations of $\relsys{U}$) and thus also 
$K$ colouring of $r$-tuples of $\relsys{S}$.  Subsequently, using Lemma \ref{sets}, there is tuple $\vec{v}\in \relsys{S}$, such that $\vec{v}=(u_1,v_1,u_2,v_2,\ldots, u_n,v_n)$ and the relations added by lift $\K$ are equivalent on $\Phi(u_1,u_2,\ldots, u_n)$ and $\Phi(v_1, v_2,\ldots, v_n)$. Thus $\relsys{U}$ induce on both sets $\{\Phi(u_1),\Phi(u_2),\ldots, \Phi(u_n)\}$ and $\{\Phi(v_1),\Phi(v_2),\ldots, \Phi(v_n)\}$ same lift $\relsys{X}$. ($\relsys{X}$ is a lift of relational structure induced by $\relsys{F}$ on $C$.)   From genericity of the $\relsys{U}$ this partial isomorphism extends to automorphism $\Psi$ of $\relsys{U}$. From the construction of relational system $\relsys{D}$ this mapping $\Psi$ sends a root of image of piece $\Piece_1$ to corresponding roots of image of piece $\Piece_2$ and thus shadow of $\relsys{U}$ contains copy of $\relsys{F}\in \F$, a contradiction.
\end{proof}
\section {Special cases of small arities}
\label{monadic}

By Theorem \ref{aritythm} it follows that only minimal classes of
finite relational structures $\F$ such that class $\Forb(\F)$ has a monadic
lift forming an amalgamation class are precisely classes $\F$ such that
all minimal separating g-cuts of their Gaifman graph have size 1.
Examples forming an amalgamation class include graphs with all its blocks being complete graphs.

Consider even more restricted classes $\F$ of structures consisting from
(relational) trees only (relational trees are formally defined bellow). In this case we can claim a much stronger result:
there exists a finite  universal object $\relsys{D}$ which is a retract of an
universal structure $\relsys{U}$.  This gives a new proof of main theorem of \cite{Nesetril2000}
and also puts our main theorem in a new context.

\subsection{Finite Dualities and Constraint Satisfaction Problems}

\label{Duality}
A Constraint Satisfaction Problem (CSP) is the following decision problem:

\smallskip

\noindent
Instance: A finite structure $\relsys{A}$

\noindent
Question: Does there exists a homomorphism $\relsys{A}\to\relsys{H}$?
\smallskip

We denote by $\CSP(\relsys{H})$ the class of all finite structures $\relsys{A}$ with $\relsys{A}\to\relsys{H}$.
It is easy to see that class $\CSP(\relsys{H})$ coincides with a particular instance of lifts and shadow. 

A {\em finite duality} (for structures of given type) is any equation
$$\Forb(\F)=CSP(\relsys{D})$$
where $\F$ is s finite set \cite{Nesetril1978,Nesetril2000,Hell2004}. $\relsys{D}$ is called {\em dual of $\F$}. We also write $\relsys{D}_\F$ for dual of $\F$ (it is easy to see that $\relsys{D}_\F$ is up to homomorphism equivalence uniquely determined). The pair $(\F,\relsys{D})$ is called {\em dual pair}.  In a sense duality is a simple constraint satisfaction problem: the existence of homomorphism into $\relsys{D}$ (i.e. a $\relsys{D}$-colouring) is equivalently characterised by a finite set of forbidden substructures. Dualities play a role not only in the complexity problems but also in logic, model theory, partial orders and categories. Particularly it follows from \cite{Atserias2008} and \cite{Rossman2005} that dualities coincide with those first order definable classes which are homomorphism closed. 

Finite dualities for monadic lifts include all classes $CSP(\relsys{H})$. We formulate this as follows:

\begin{prop}
\label{prop21}
For a class $\K$ of structures the following two statements are equivalent:
\begin{enumerate}
\item $\K=\CSP(\relsys{H})$ for finite $\relsys{H}$.
\item There exists a class $\K'$ of monadic lifts such that:
\begin{enumerate}
\item The shadow of $\K'$ is $\K$,
\item $\K'=\Forb(\F')\cap \Forbi(\relsys{K}_1)$ where $\F'$ is a finite set of monadic covering lifts of edges (i.e. every $\relsys{F}\in\F'$ contains at most one non-unary tuple). $\Forbi(\relsys{K}_1)$ means that every vertex belongs to a unary lifted tuple.
\end{enumerate}
\end{enumerate}
\end{prop}

\begin{proof}[Proof (sketch)]
1. obviously imply 2.

In the opposite direction construct $\relsys{H}$ as follows:  Let
$\relsys{H}_0$ be lift with vertex for every consistent combination of new
relations  $\ext{H_0}{i}$ and relations $\rel{H_0}{i}$ empty.  Now construct lift
$\relsys{H}$ on same vertex set as $\relsys{H}_0$ with $\ext{H}{i}=\ext{H_0}{i}$.
Put tuple $\vec{x}\in \rel{H}{i}$ if and only if the structure induced by $\vec{x}$ is on
$\relsys{H}_0$ with $\vec{x}$ added to $\rel{H}{i}$ is in $\Forb(\F')$.
Consequently if $\Age(\Forb(\F'))$ is an amalgamation class then $\Age(Forb(\F))$ is an amalgamation class too.
\end{proof}

In the language of dualities this amounts to saying that the classes $\CSP(\relsys{H})$ are just classes described by shadow dualities of the simplest kind: forbidden lifts are just vertex coloured edges.

A (relation) tree can be defined as follows: The {\em incidence graph} $ig(\relsys{A})$ of relational structure $\relsys{A}$ is the                      
bipartite graph with parts $A$ and                                  
$$                                                                              
Block(A) = \{ (i,(a_1, \ldots, a_{\delta_i})) : i \in I,                       
(a_1, \ldots, a_{\delta_i}) \in \rel{A}{i} \},                                      
$$                                                                              
and edges $[a, (i,(a_1, \ldots, a_{\delta_i}))]$ such that                      
$a \in (a_1, \ldots, a_{\delta_i})$. (Here we write                             
$x \in (x_1, \ldots, x_n)$ when there exists an index $k$                       
such that $x = x_k$; $Block(A)$ is a multigraph.)  $\relsys{A}$ is called a {\em tree} when                 
$ig(\relsys{A})$ is a graph tree (see e.g. \cite{Matousek1998}).
The definition of relational trees by the incidence graph $ig(\relsys{A})$ allows us to use graph terminology for relation trees. 

Finite dualities have been characterised (see \cite{Nesetril2000}):
\begin{thm}
For every type $\Delta$ and for every finite set $\F$ of finite relational trees there exists a dual $\Delta$-structure $\relsys{D}_\F$. Up to homomorphism equivalence there are no other dual pairs.
\end{thm}
\label{dualpairs}
A number of constructions of duals for given set $\F$ are known \cite{Nesetril2005}.
It follows from this section that we have a yet another approach to this problem:
\begin{corollary}
\label{nasdual}
Let $\F$ be a set of finite tree-structures, then there exists a finite set of lifted structures $\F'$ with the following properties:
\begin{itemize}
  \item[$(i)$] $\Age(\Forbi(\F'))$ is an amalgamation class (and thus there is universal $\relsys{U}'\in \Forbi(\F')$),
  \item[$(ii)$] all lifts in $\Forbi(\F')$ are monadic,
  \item[$(iii)$] $\sh(\relsys{U'})=\relsys{U}$ is universal for $\K$,
  \item[$(iv)$] $\relsys{U}'$ has a finite retract $\extsys{D}_\F$ and consequently $\sh(\extsys{D}_\F)=\relsys{D}_\F$ is a dual of $\F$.
\end{itemize}
\end{corollary}

\begin{proof}[Proof]
Observe that the inclusion minimal g-cuts of a relational tree
are all size 1. Thus for fixed family $\F$ of finite relational trees
 our Theorem \ref{mainthm} establishes the existence monadic lift
that give an generic structure $\relsys{U}'$ whose shadow is (homomorphism) universal for $\Forb(\F)$.

This structure $\relsys{U}'$ is countable.  To get a dual, we find finite
$\relsys{D'}\in \Lifts$ which is retract of $\relsys{U}'$ and still there is homomorphism $\relsys{Y}\to \relsys{U}'$ if and only if there is homomorphism $\relsys{Y} \to \relsys{X'}$.
Construct $\relsys{D}'$ by unifying all
vertices of the same colour.   This gives a homomorphism (retraction) $r:\sh(\relsys{U}')\to
\relsys{D}'$. Put $\relsys{D}=\sh(\relsys{D}')$. We show that $\relsys{D}$ is the dual of $\F$.

For every $\relsys{A}\in \Forb(\F)$ there is an embedding
$e:\relsys{A}\to\sh(\relsys{U}')$.  It follows that $\relsys{A}\in
\CSP(\{\relsys{D}\})$ because $e\circ r$ is a homomorphism $\relsys{A}\to
\relsys{D}$.  To see that $\relsys{D}\notin \Forb(\F)$, assume for a
contradiction that there is $\relsys{F}\in \F$ and a homomorphism
$f:\relsys{F}\to \relsys{D}$. Create lift $\relsys{X}$ from $\relsys{F}$ by
adding a tuple $(v)$ to $\ext{X}{i}$ if and only if $(f(v))\in \ext{D'}{i}$.
Lift $\relsys{X}$ satisfy condition $(a)$ of the definition of $\Lifts$.  For
every $\relsys{F}'\in \F$ a homomorphism $\relsys{F}'\to \relsys{X}$ implies a
homomorphism $\relsys{F}'\to\relsys{U}$ giving (b) and thus $\relsys{X}\in
\Lifts$. A contradiction with $\relsys{U}'$ being generic for $\Lifts$ and
$\sh(\relsys{U}')\in \Forb(\F)$.

\end{proof}

Note that it is also possible to construct $\relsys{D}_\F$ in a finite way
without using the \Fraisse{} limit: for every possible combination of new
relations on single vertex create single vertex of $\relsys{D}_\F$ and then
keep adding tuples as long as possible so $\relsys{D}_\F$ is still in $\Lifts$.

Finally, let us remark that one can also prove that $\extsys{U}$ has a finite presentation in the sense of $\cite{Hubicka2005a}$.

\label{aplikace}

The situation of course changes if we consider more complicated (non-monadic and non-trees) lifted classes $\Forbi(\F')$ and in fact it has been proved in \cite{Kun2008} that any NP language is polynomial-time equivalent to a class $\sh(\Forbi(\F'))$ for a finite set $\F'$ set of lifted structures. As these structures may be supposed to be connected
(by adding dummy edges) we have the following corollary of Theorem
\ref{univthm}:
\begin{corollary}
\label{NPU}
Every NP problem is polynomial-time equivalent to membership problem for a class $\Age(\relsys{U})$ where $\relsys{U}$ is shadow of an ultrahomogeneous structure $\extsys{U}$.
\end{corollary}

Constraint Satisfaction Problems $CSP(\relsys{H})$ for countable templates $\relsys{H}$ were investigated in \cite{Bodirsky2003} and it has been shown there that the basic results of so called ``algebraic method'' \cite{Bulatov2005} hold also for $\omega_0$-categorical templates \cite{Bodirsky2003}. Recently, \cite{Bodirsky2011a} proved that $CSP(\relsys{H})$ for general countable template $\relsys{H}$ encodes any problem in NP. It is not known whether $\relsys{H}$ may be chosen $\omega_0$-categorical. Corollary \ref{NPU} complements this: lifts and shadows enable to code any NP language by membership problem for ages of shadows of ultrahomogeneous structures.

\subsection {Forbidden cycles and Urysohn spaces (binary lifts)}

We turn briefly our attention to binary lifts.  This relates some of the earliest results on universal graphs with (recently intensively studied)
the Urysohn spaces:

\label{urysohnsection}
We will consider a finite family $\F$ consisting of graphs of odd cycles of lengths
$3,5,7,\ldots, l$.  As shown by \cite{Komjath1988} (see also \cite{Cherlin1999}) those families have universal
graphs in $\Forbi(\F)$ and as shown by \cite{Cherlin1996} those are the only
classes defined by forbidding a finite set of  cycles.
These classes also form especially easy families of pieces.  In fact each piece
is an undirected path of length at most $l$, where $l$ is the length of longest cycle
in $\F$ with both ends of the patch being roots.  
Applying our lift construction introduced in Section~\ref{sec:basis} we obtain particularly easy
description of the lifted structure.

We can proceed as follows: For vertices $x$, $y$ of graph $G$ put $d(x,y)=(a,b)$
if $a$ (b respectively) is the shortest length of an even (odd respectively) length
of walk from $x$ to $y$.

We call $d$ an {\em even-odd distance} and one can develop a theory analogous to the theory
of metric embeddings. As graphs with odd cycles of length $\leq l$ are just graphs
for which there are no $x$, $y$ with $d(x,y)=(a,b)$ and $a+b\leq l$, we obtain
a common generalisation of universal odd girth graphs and Urysohn type spaces. 

We use the following definition which is motivated by metric spaces.
When specialised to graphs, this definition is analogous to (corrected form) of an
$s$-structure \cite{Komjath1988}. Our approach also gives new easy description (i.e. finite presentation)
of the Urysohn space \cite{Hubicka2008}.  

\begin{defn}
\label{evenoddp}
Pair $(a,b)$ is considered to be {\em even-odd pair} if $a$ is even non-negative integer or $\omega$ and $b$ is odd non-negative integer or $\omega$.

For even-odd pairs $(a,b)$ and $(c,d)$ we say that $(a,b)\leq (c,d)$ if and only if $a\leq c$ and $b\leq d$.  Consider $a+\omega=\omega$ and $\omega+b=\omega$.  Put: $$(a,b)+(c,d) = (\min(a+c,b+d),\min (a+d,b+c)).$$

For a set $S$, function $d$ from $S$ to even-odd pairs is called {\em even-odd distance function on $S$} if following conditions are satisfied:
\begin{enumerate}
\item $d(x, y) = (0,b)$   if and only if   $x = y$,
\item $d(x, y) = d(y, x)$,
\item $d(x, z) \leq d(x, y) + d(y, z)$.

\end{enumerate}
Finally pair $(S,d)$ where $d$ is even-odd distance function for $S$ is called {\em even-odd metric space}.
\end{defn}

Note that the even-odd metric spaces differ from usual notion of metric space
primarily by the fact that the ordering of values of distance function is not
linear, but forms a 2-dimensional partial order.  
But some basic results about metric spaces are valid even in this setting.

As noted above, even-odd metric space form a stronger version of the distance metric on the
graph.  (For graph $\relsys{G}$ we can put $d(x,y)=(a,b)$ where $a$ is length of the
shortest walk of even length connecting $x$ and $y$, while $b$
is the length of shortest walk of odd length.)

The even-odd distance metric specifies length of all possible walks: for a graph
$\relsys{G}$ and a even-odd distance metric $d$ we now have a walk connecting $x$ and
$y$ of length $a$ if and only if $d(x,y)=(b,c)$ having $b\leq a$ for $a$ even or $c\leq
a$ for $a$ odd.

It is well known that universal and homogeneous metric space exists for several classes of
metric spaces \cite{Delhomme2007,The2006}.  Analogously we have:
\begin{thm}
\label{Uthm}
There exists a generic even-odd metric space $\relsys{U}$.
\end{thm}
\begin{proof}
 We prove that class $\M$ of all even-odd metric spaces is an amalgamation class.

To show that $\M$ has the amalgamation property, 
 take a free amalgam $\relsys{D}$ of even-odd metric spaces $\relsys{A}$,
$\relsys{B}$ over $\relsys{C}$. This amalgam is not even-odd metric space, since some distances are not defined.

We can however define a walk from $v_1$ to $v_t$ of length $l$ in $\relsys{D}$ as a sequence of vertices $v_1,v_2,v_3,\ldots, v_t$
and distances $d_1,d_2,d_3,\ldots, d_{t-1}$ such that $\sum_{i=1}^{t-1} d_i=l$ and $d_i$ is present in
the even-odd pair $d(v_i,v_i+1)$ for $i=1,2,\ldots,t-1$.

We then produce even-odd metric space $\relsys{U}$ on same vertex set as
$\relsys{D}$ where distance of some vertices $a,b\in E$ is even-odd pair
$(l,l')$ such that $l$ is the smallest even value such that there exists walk
joining $a$ and $b$ of length $l$ in $\relsys{D}$. $l'$ is the smallest odd value such
that there exists walk from $a$ to $b$ of length $l'$.

It is easy to see that $\relsys{U}$ is even-odd metric space (every triangular
inequality is supported by the existence of a walk) and other properties of
amalgamation class follows from definition.
\end{proof}

The graphs omitting odd cycles up to length $l$ can be axiomatised by simple
condition on their even-odd distance metric: there are no vertices $x,y$ such
that $d(x,y)=(a,b)$ where $a+b\leq l$. Denote by $\K_l$ the class of all countable even-odd metric spaces such that there are no vertices $x,y$ such that $d(x,y)=(a,b)$ with $a+b\leq l$.  The existence of universal and homogeneous even-odd metric space $\relsys{U}_l=(U_l,d_l)$ for the class $\K_l$ is simple consequence of Theorem \ref{Uthm}. In fact $\relsys{U}_l$ is a subspace of $\relsys{U}$ induced by all those vertices $v$ of $\relsys{U}$ satisfying $d(v,v)=(0,b)$ and $b>l$.
On the other hand $\relsys{U}_l$ is induced by a graphs:

\begin{thm}
For metric space $\relsys{U}_l=(U_l,d_l)$ denote by $\relsys{G}_l=(U_l,E_l)$ a graph on vertex set $U_l$ where $\{x,y\}\in E_l$ if and only if $d(x,y)=(a,1)$.

For every choice of odd integer $l\geq 3$, $\relsys{G}_l$ is universal graph for class of all graphs omitting odd cycles of length at most $l$.
\end{thm}
\begin{proof}
Graph $\relsys{G}_l$ is omitting odd cycles up to length $l$ from the fact that any two
vertices $x,y$ on a odd cycle of length $k$ have distance $d(x,y)=(a,b)$ where
$a+b$ is at most $l$.

Now consider any countable graph $\relsys{G}=(V,E)$ omitting odd cycles of length at most
$l$. Construct corresponding even-odd distance metric space $(V,d_\relsys{G})$. By universality argument  $(V,d_\relsys{G})$ is subspace of $\relsys{U}_l$ and thus also $\relsys{G}$ subgraph of $\relsys{G}_l$.
\end{proof}

The explicit construction of the Urysohn space $\relsys{U}$ without using \Fraisse{} limit
argument, is described in \cite{Hubicka2008} and same technique can be carried to
even-odd metric. This is captured by the following definition.

\begin{defn}
\label{defU}
The vertices of $\U$ are functions $f$ such that:
\begin{enumerate}
\item[(1)] The domain $D_f$ of $f$ is finite (possibly empty) set of functions and $\emptyset$.
\item[(2)] The range of $f$ are even-odd pairs.
\item[(3)] For every $g\in D_f$ and $h\in D_{g}$, we have $h\in D_f$.
\item[(4)] $D_f$ using metric $d_{\U}$ defined bellow forms a even-odd metric
space.
\item[(5)] $f$ defines an extension of even-odd metric space on vertices $D_f$ by adding a new vertex. This means that $f(\emptyset)=(0,x)$ and for every
$g, h\in D_f$ we have $f(g)+f(h)\leq d_{\U}(g,h)$ and $f(g)\geq f(h)+d_\U(g,h)$.

\end{enumerate}
The metric $d_{\U}(f,g)$ is defined by:
\begin{enumerate}
  \item if $f=g$ then
$d_{\U}(f,g)=f(\emptyset),$
  \item if $f\in D_g$ then
$d_{\U}(f,g)=g(f),$
  \item if $g\in D_f$ then 
$d_{\U}(f,g)=f(g),$
  \item if none of above holds then
$d_{\U}(f,g)=min_{h\in D_f\cap D_g} f(h)+g(h).$

Minimum is taken element-wise on the pairs.
\end{enumerate}
 \end{defn}

\begin{thm}\label{generic}
$(\U,d_\U)$ is the generic even-odd metric space.
\end{thm}
Proof is analogous to \cite{Hubicka2008}.
This may be seen as yet another incarnation of Kat\v etov functions \cite{Uspenskij2008}, see also \cite{The2006}.

We shall also remark that the universal structures for $\F=\{C_l\}$ can also be constructed as metrically
homogeneous graphs, see \cite{Cherlin}. 

\bibliography{homuniv-restricted.bib}
\end{document}